\documentclass{article}
\usepackage[T1]{fontenc}
\usepackage{amsmath,amssymb,amsthm,mathrsfs,calc,graphicx,pdfpages,bm,tikz}
\usepackage[left=3.0cm, top=3.5cm,bottom=3.5cm,right=3.0cm]{geometry}

 \usepackage{mathpazo}
 
\newtheorem{theorem}{Theorem}[section]
\newtheorem{lemma}[theorem]{Lemma}

\theoremstyle{definition}

\theoremstyle{remark}

\numberwithin{equation}{section}

\DeclareMathOperator{\conv}{conv}

\newcommand{\Var}{\operatorname{Var}}

\def\ba{\begin{array}}
\def\ea{\end{array}}
\def\bea{\begin{eqnarray} \label}
\def\eea{\end{eqnarray}}
\def\be{\begin{equation} \label}
\def\ee{\end{equation}}
\def\bit{\begin{itemize}}
\def\eit{\end{itemize}}
\def\ben{\begin{enumerate}}
\def\een{\end{enumerate}}

\def\EE{\mathbb{E}}

\def\NN{\mathbb{N}}
\def\PP{\mathbb{P}}

\def\RR{\mathbb{R}}
\def\SS{\mathbb{S}}

\def\b{\beta}
\def\g{\gamma}

\def\G{\Gamma}

\def\dint{\textup{d}}

\def\Vis{\textup{Vis}}

\allowdisplaybreaks

\begin{document}

\title{\textbf{Random polytopes: central limit theorems\\for intrinsic volumes}}

\author{Christoph Th\"ale\footnotemark[1],\; Nicola Turchi\footnotemark[2]\;\footnotemark[3]\; and Florian Wespi\footnotemark[4]}

\date{}
\renewcommand{\thefootnote}{\fnsymbol{footnote}}

\footnotetext[1]{
Faculty of Mathematics, Ruhr University Bochum, Germany. Email: \texttt{christoph.thaele@rub.de}}

\footnotetext[2]{
Faculty of Mathematics, Ruhr University Bochum, Germany. Email: \texttt{nicola.turchi@rub.de}}

\footnotetext[4]{Institute of Mathematical Statistics and Actuarial Science, University of Bern, Switzerland.\\Email: \texttt{florian.wespi@stat.unibe.ch}}

\footnotetext[3]{Research supported by the Research Training Group RTG2131 \textit{High-dimensional Phenomena in Probability -- Fluctuations and Discontinuity}.}

\maketitle

\begin{abstract}
\noindent 	Short and transparent proofs of central limit theorems for intrinsic volumes of random polytopes in smooth convex bodies are presented. They combine different tools such as estimates for floating bodies with Stein's method from probability theory.
\end{abstract}
\maketitle
\section{Introduction and main result}
Fix a space dimension $n\geq 2$, let $N\geq n+1$ and suppose that $X_1,\ldots,X_N$ are independent random points that are uniformly distributed in a prescribed convex body $K$, which we assume to have a boundary which is twice differentiable and has positive Gaussian curvature everywhere. The convex hull of $X_1,\ldots,X_N$ is denoted by $K_N$. In this note we are interested in the intrinsic volumes $V_j(K_N)$ of $K_N$, $j\in\{1,\ldots,n\}$. These functionals are of particular importance in convex and integral geometry since they (together with the Euler-characteristic) form a basis of the vector space of all motion invariant and continuous valuations on convex bodies according to Hadwiger's celebrated theorem, cf.\ \cite{Schneider}. The purpose of this text is to prove central limit theorems for $V_j(K_N)$, $j\in\{1,\ldots,n\}$, as $N\to\infty$. Let us point out that such a result is not totally new. Central limit theorems for general $j\in\{1,\ldots,n\}$ were known for a long time only for the Poisson setting in the the special case that $K$ is the $n$-dimensional Euclidean unit ball, see the paper of Calka, Schreiber and Yukich \cite{CalkaSchreiberYukich}. We also refer to the paper of Schreiber \cite{Schreiber} for the case $j=1$. Only very recently (in parallel and independently of us) Lachi\`eze-Rey, Schulte and Yukich \cite{LachiezeReySchulteYukich} gave a proof for the general case by embedding the problem into the theory of so-called stabilizing functionals.

Using estimates for floating bodies, in combination with a general normal approximation bound obtained by Chatterjee \cite{Chatterjee} and Lachi\`eze-Rey and Peccati \cite{LachiezeReyPeccati} originating in Stein's method, our contribution is a quick, transparent and direct proof of the central limit theorems for the intrinsic volumes $V_j(K_N)$, $j\in\{1,\ldots,n\}$, as $N\to\infty$. More precisely, while the traditional methods (see \cite{Pardon,Reitzner05,VuCLT}) first use a conditioning argument to compare \(K_N\) with the floating body and to prove the central limit theorem for a Poissonized version of the random polytopes, before pushing this result to the original model by de-Poissonization, we give a direct proof without making the detour just described. We also avoid this way the more technical theory of stabilizing functionals developed in \cite{LachiezeReySchulteYukich}.

To present our result formally, we shall use the notation $a_N\ll b_N$ whenever for two sequences $(a_N)$ and $(b_N)$, $a_N\leq c\,b_N$ for sufficiently large $N\geq n+1$ and some constant $c\in(0,\infty)$ not depending on $N$ (but possibly on the space dimension $n$ and the convex body $K$).


We define the Wasserstein distance between two random variables \(X\) and \(Y\) as
\begin{equation}\label{eq:Wasserstein}
\mathrm{d}(X,Y) :=\sup_{h\in\mathrm{Lip}_1}\bigl|\EE h(X)-\EE h(Y)\bigr|\,,
\end{equation} 
where the supremum is running over all Lipschitz functions $h:\RR\to\RR$ with Lip\-schitz constant less or equal than \(1\). It is well-known that if $G$ is a standard Gaussian random variable and if $(W_N)_{N\in\NN}$ is a sequence of centred random variables with finite second moment such that $\mathrm{d}(W_N/\sqrt{\Var W_N},G)$ $\to 0$, as $N\to\infty$, then $W_N$ converges in distribution to $G$.

We take into consideration the intrinsic volumes of \(K_N\)
\begin{equation*}
V_j(K_N)\,,\quad j\in\{1,\ldots,n\}\,,\quad N\geq n+1\,,
\end{equation*}
for which we prove the following central limit theorems. As discussed earlier, this extends the results in \cite{Schreiber,VuCLT} to more general convex bodies and to arbitrary intrinsic volumes. 

\begin{theorem}\label{thm:CLT}
	Let $K\subset\RR^n$ be a convex body with twice differentiable boundary and strictly positive Gaussian curvature everywhere. Then, for all $j\in\{1,\ldots,n\}$, one has that $(V_j(K_N)-\EE V_j(K_N))/\sqrt{\Var V_j(K_N)}$ converges in distribution to a standard Gaussian random variable \(G\), as $N\to\infty$. More precisely,
	\begin{equation*}
	\mathrm{d}\biggl(\frac{V_j(K_N)-\EE V_j(K_N)}{\sqrt{\Var V_j(K_N)}},G\biggr)\ll N^{-\frac{1}{2}+\frac{1}{n+1}}\,(\log N)^{3+\frac{2}{n+1}}\,.
	\end{equation*}
\end{theorem}
The rest of this note is structured as follows. In Section \ref{sec:Background} we collect some background material in order to keep our presentation reasonably self-contained. The proof of the central limit theorems is the content of the final Section \ref{sec:ProofCLT}.
\section{Background material}\label{sec:Background}
\subsection*{Convex bodies.}
By a convex body we understand a compact convex subset of $\RR^n$ that has non-empty interior. 
We use the symbol $B^n$ to denote the centred Euclidean unit ball in $\RR^n$. For points $X_1,\ldots,X_N\in\RR^n$, we write $[X_1,\ldots,X_N]$ to indicate their convex hull.
\subsection*{Intrinsic volumes.}
Let $K\subset\RR^n$ be a convex body and fix $j\in\{0,\ldots,n\}$. We denote by $G(n,j)$ the Grassmannian of $j$-dimensional linear subspaces of $\RR^n$, which is supplied with the unique Haar probability measure $\nu_j$, see \cite{Schneider}. For $L\in G(n,j)$, we write $\mathrm{vol}_j(K|L)$ for the $j$-dimensional Lebesgue measure of the orthogonal projection of $K$ onto $L$. Finally, we let for integers $\ell\in\NN$, $\kappa_\ell=\pi^{\ell/2}\Gamma(1+\frac{\ell}{2})^{-1}$ be the volume of the $\ell$-dimensional unit ball. Then the $j$th intrinsic volume of $K$ can be defined as
\begin{equation}\label{eq.Kubota}
V_j(K) := \binom{n}{ j}\frac{\kappa_n}{\kappa_j\kappa_{n-j}}\int_{G(n,j)}\mathrm{vol}_j(K|L)\,\nu_j(\dint L)\,.
\end{equation}
For example, $V_n(K)$ is the ordinary volume (Lebesgue measure), $V_{n-1}(K)$ is half of the surface area, $V_1(K)$ is a constant multiple of the mean width and $V_0(K)\equiv 1$ is the Euler-characteristic of $K$.
\subsection*{An estimate for subspaces.}
Fix $j\in\{1,\ldots,n-1\}$, $L\in G(n,j)$ and $z\in\SS^{n-1}$. The angle $\sphericalangle(z,L)$ between $z$ and $L$ is defined as the minimum angle $\min\{\sphericalangle(z,x):x\in L\}$. We now recall the following fact from \cite[Lemma 1]{BaranyFodorVigh}, see also \cite[Lemma 10]{BaranyThaele}.
\begin{lemma}\label{lem:SubspacesAngle}
	We have that
	\begin{align*}
	\nu_j(\{L\in G(n,j):\sphericalangle(z,L)\leq a\}) \ll a^{n-j}
	\end{align*}
	for all sufficiently small $a>0$.
\end{lemma}
\subsection*{Floating bodies.}
We recall the concept of the floating body, that was introduced independently in \cite{BaranyLarman} and \cite{SchuettWerner}. Let $K\subset\RR^n$ be a convex body and $t>0$ (we shall implicitly assume that $t$ is sufficiently small). If $H$ is a half space of $\RR^n$ with $V_n(K\cap H)=t$, the set $K\cap H$ is called a $t$-cap of $K$. The union of all these $t$-caps is the so-called wet part of $K$ and its complement is the $t$-floating body of $K$. In what follows, we shall denote the $t$-floating body by $K_{(t)}$.
We  rephrase a result of B\'ar\'any and Dalla \cite{BaranyDalla}, which has also been proved by Vu \cite{VuConcentration} using different techniques, see also Lemma 2.2 in \cite{ReitznerSurvey}. Recall that $K_N$ is the random polytope generated by $N$ independent random points that are uniformly distributed in a convex body $K\subset\RR^n$ having sufficiently smooth boundary (in fact, smoothness of the boundary is not needed in the next lemma).
\begin{lemma}\label{prop:VanVu}
	For any $\b\in(0,\infty)$, there exists a constant $c=c(\b,n)\in(0,\infty)$ only depending on $\b$ and on $n$ such that the probability of the event that $K_N$ does not contain the \(c\,\frac{\log N}{N}\text{-floating}\) body is at most $N^{-\b}$, whenever $N$ is sufficiently large.
\end{lemma}
\subsection*{A general bound for normal approximation.}
Let $S$ be a Polish space and $f:\bigcup_{k=1}^{N} S^k\to\RR$ be a measurable and symmetric function acting on point configurations of at most $N\in\NN$ points in $S$. For $x=(x_1,\ldots,x_N)\in S^N$, we write $x^i$ for the vector $x$ with the $i$th coordinate removed and $x^{i_1i_2}$ for the vector that arises from $x$ by removing coordinates $i_1$ and $i_2$. Next, we define the first- and second-order difference operator applied to $f(x)=f(x_1,\ldots,x_N)$ by
\begin{align*}
D_if(x) := f(x) - f(x^i)\quad\text{and}\quad D_{i_1,i_2}f(x) := f(x) - f(x^{i_1}) - f(x^{i_2}) + f(x^{i_1i_2}),
\end{align*}
respectively.

We denote by $X=(X_1,\ldots,X_N)$ a random vector of elements of $S$. Let $X',\widetilde{X}$ be independent copies of $X$ and say that a random vector $Z=(Z_1,\ldots,Z_N)$ is a recombination of $\{X,X',\widetilde{X}\}$ provided that $Z_i\in\{X_i,X_i',\widetilde{X}_i\}$ for all $i\in\{1,\ldots,N\}$.

To rephrase the normal approximation bound from \cite{Chatterjee} in the form that can be deduced from \cite{LachiezeReyPeccati} we define
\begin{align*}
\gamma_1 & :=\sup_{(Y,Y',Z,Z')}\EE\bigl[\mathbf{1}\{D_{1,2}f(Y)\neq 0\}\,\mathbf{1}\{D_{1,3}f(Y')\neq 0\}\,(D_2f(Z))^2\,(D_3f(Z'))^2\bigr]\,,\\
\gamma_2 & := \sup_{(Y,Z,Z')}\EE\bigl[\mathbf{1}\{D_{1,2}f(Y)\neq 0\}\,(D_1f(Z))^2\,(D_2f(Z'))^2\bigr]\,,\\
\gamma_3 & := \EE|D_{1}f(X)\big|^4\,,\\
\gamma_4 &:= \EE|D_1f(X)|^3\,,
\end{align*}
where the suprema in the definitions of $\gamma_1$ and $\gamma_2$ run over all quadruples or triples of vectors $(Y,Y',Z,Z')$ or $(Y,Z,Z')$ that are recombinations of $\{X,X',\widetilde{X}\}$, respectively. Next, we define $W:=f(X_1,\ldots,X_N)$ and assume that $\EE W=0$ and $0<\EE W^2<\infty$. Recall that $\mathrm{d}(\,\cdot\,,\,\cdot\,)$ denotes the Wasserstein distance defined in \eqref{eq:Wasserstein}.

\begin{lemma}\label{lem:SecondOrderPoincare}
	Under the assumptions stated above, if $G$ denotes a standard Gaussian random variable, then
	\begin{equation*}
	\mathrm{d}\biggl(\frac{W}{\sqrt{\Var W}},G\biggr)\ll \frac {\sqrt{N}}{\Var\,W} \,\Bigl(\sqrt{N^2\gamma_1}+\sqrt{N\gamma_2}+\sqrt{\g_3}\,\Bigr)+\frac{N}{(\Var\,W)^{\frac{3}{2}}}\,\gamma_4\,.
	\end{equation*}
\end{lemma}

\section{Proof of Theorem \ref{thm:CLT}}\label{sec:ProofCLT}

In the proof of our result we will make use of the following lower and upper variance bounds, proven by B\'ar\'any, Fodor and V\'igh \cite{BaranyFodorVigh}, namely,
\begin{equation}\label{eq:LowerVarianceBound}
N^{-\frac{n+3}{n+1}}\ll\Var V_j(K_N)\ll N^{-\frac{n+3}{n+1}}
\end{equation}
for all $j\in\{1,\ldots,n\}$

According to Lemma \ref{prop:VanVu}, we see that for any $\b\in(0,\infty)$ there exists a constant $c=c(\b,n)\in(0,\infty)$ such that the random polytope $[X_2,\dots,X_N]$ contains the floating body $K_{(c\log N/N)}$ with high probability. More precisely, denoting the latter event by $B_1$, it holds that for sufficiently large $N$,
\begin{equation}\label{eq:ProbabilityAc}
\PP(B_1^c)\leq (N-1)^{-\b} \leq c_1 N^{-\b}\,,
\end{equation}
where $c_1\in(0,\infty)$ is a constant independent of $N$. Note that we choose $\b$ large enough ($\b=5$ will be sufficient for all our purposes). 

Next, we let $Y,Y',Z,Z'$ be recombinations of our random vector $X=(X_1,\ldots,X_N)$ and denote by $B_2$ the event that $\bigcap_{W\in\{Y,Y',Z,Z'\}}[W_4,\ldots,W_N]$ contains $K_{(c\log N/N)}$. By the union bound it follows that the probability of $B_2^c$ is also small:
\begin{equation}\label{eq:ProbabilityBc}
\PP(B_2^c)\leq c_2N^{-\b}\,,
\end{equation}
where $c_2\in(0,\infty)$ is again a constant independent of $N$.

\begin{proof}[Proof of Theorem \ref{thm:CLT}]
	Assume without loss of generality that $K$ has volume one. 
	The idea of the proof is to apply the normal approximation bound in Lemma \ref{lem:SecondOrderPoincare} to the random variables
	\begin{equation*}
	W=f(X_1,\ldots,X_N):=V_j([X_1,\ldots,X_N])-\EE V_j(K_N)\,,
	\end{equation*}
	which clearly satisfy the assumptions made in Lemma \ref{lem:SecondOrderPoincare}. To this end, we  need to control, in particular, the first- and second-order difference operators $D_iW=D_iV_j(K_N)$ and $D_{i_1,i_2}W=D_{i_1,i_2}V_j(K_N)$ for $i,i_1,i_2\in\{1,\ldots,N\}$.
	
	Conditioned on the event $B_1$, we use \eqref{eq.Kubota} to estimate the first-order difference operator applied to the intrinsic volume functional $V_j(K_N)$ as follows:
	\begin{equation}\label{eq:Bound1stOrderDifference}
	\begin{split}
	D_1V_j(K_N)&=\binom{n} {j}\frac{\kappa_n}{\kappa_j\kappa_{n-j}}\int_{G(n,j)}\mathrm{vol}_j((K_N|L)\setminus ([X_2,\dots,X_N]|L))\,\nu_j(\dint L)\\
	&\qquad\qquad\qquad\qquad\qquad\qquad\times\mathbf{1}\{X_1\in K\setminus K_{(c\log N/N)}\}.
	\end{split}
	\end{equation}
	For the sake of brevity we will indicate \([X_2,\dots,X_N]\) by \(K_{N-1}\).
	On the event $B_1$ we first notice that ${\mathrm {vol}}_j((K_{N}|L)\setminus (K_{N-1}|L))$ is zero if $X_1\in K_{N-1}$. So, we can restrict to the situation that $X_1\in K\setminus K_{N-1}$, which conditioned on $B_1$ occurs with probability $V_n(K\setminus K_{N-1})\ll V_n(K\setminus K_{(c\log N/N)})\ll (\log N/N)^{\frac{2}{n+1}}$, cf.\ \cite[Theorem 6.3]{BaranyLectureNotes}. 
	
	Suppose now that the convex body \(K\) is the normalized Euclidean unit ball in \(\RR^n\). It is our aim to define a full-dimensional cap $C$ such that $K_N\setminus K_{N-1}$ is contained in $C$. For this reason, we define $z$ to be the closest point to \(X_1\) on \(\partial K\) (we notice that $z$ is uniquely determined if $K_{(c\log N/N)}$ is non-empty). The visible region of $z$ is defined as
	\begin{equation*}
	\Vis_z(N) := \{x\in K\setminus K_{(c\log N/N)}:[x,z]\cap K_{(c\log N/N)}=\emptyset\}\,.
	\end{equation*}
	By definition of the floating body $K_{(c\log N/N)}$, the diameter of $\Vis_z(N)$ is $c_3(\log N/N)^{\frac{1}{n+1}}$, where $c_3\in(0,\infty)$ is a constant not depending on $N$. Let us denote by $D(z,c_3(\log N/N)^{\frac{1}{n+1}})$ the set of all points on the boundary of $K$ which are of distance at most $c_3(\log N/N)^{\frac{1}{n+1}}$ to $z$. Then, it follows from \cite[Lemma 6.2]{VuConcentration} that $C:=\conv\{D(z,c_3(\log N/N)^{\frac{1}{n+1}})\}$ has volume of order at most $\log N/N$. Moreover, $C$ is in fact a spherical cap and the central angle of it is denoted by $\alpha$. For a subspace $L\in G(n,j)$, one has that $(K_N|L)\setminus (K_{N-1}|L)\subseteq C|L$.
	The volume $\mathrm{vol}_j(C|L)$ of the projected cap $C|L$ is $\mathrm{vol}_j(C|L)\ll (\log N/N)^{\frac{j+1}{n+1}}$. Indeed, the height of \(C|L\) keeps the order of the height of \(C\), namely \((\log N/N)^{\frac{2}{n+1}}\), while the order of its base changes from \(((\log N/N)^{\frac{1}{n+1}})^{n-1}\) to  \(((\log N/N)^{\frac{1}{n+1}})^{j-1}\), since $L$ is a subspace of dimension $j$. Note that, by construction of \(C\), if $\sphericalangle(z,L)$, the angle between $z$ and $L$, is too wide compared to \(\alpha \) , then \(C|L\subseteq K_{N-1}|L\), for sufficiently large \(N\). In particular, \((K_{N}\setminus K_{N-1})|L \subseteq K_{N-1} |L\), which implies \(K_{N}|L=K_{N-1}|L\). In fact, it is easily checked that the integrand in \eqref{eq:Bound1stOrderDifference} can only be non-zero if \(\sphericalangle(z,L)\ll\alpha\) (the constant can be taken to be \(2\) in the case of the ball). Therefore, we can restrict the integration in \eqref{eq:Bound1stOrderDifference} to the set $\{L\in G(n,j):\sphericalangle(z,L)\ll\alpha\}$. It is not difficult to verify that $\alpha\ll V_n(C)^\frac{1}{n+1}$, see also Equation (27) in \cite{BaranyFodorVigh}. 
	
	Taken all together, this yields
	\begin{equation*}
	\begin{split}
	D_1V_j(K_N) \ll \biggl({\frac{\log N}{N}}\biggr)^{\frac{j+1}{n+1}}&\,\nu_j\Bigl(\Bigl\{L\in G(n,j):\sphericalangle(z,L)\ll V_n(C)^\frac{1}{n+1}\Bigr\}\Bigr)\\
	&\times\mathbf{1}\{X_1\in K\setminus K_{(c\log N/N)}\}\,.
	\end{split}
	\end{equation*}
	According to Lemma \ref{lem:SubspacesAngle} and the fact that $V_n(C)\ll \log N/N$, it holds that
	\begin{equation*}
	\nu_j\Bigl(\Bigl\{L\in G(n,j):\sphericalangle(z,L)\ll V_n(C)^\frac{1}{n+1}\Bigr\}\Bigr)\ll \biggl({\frac{\log N}{N}}\biggr)^{\frac{n-j}{n+1}}\,,
	\end{equation*}
	which in turn implies
	\begin{equation}\label{eq:D1bound}
	\begin{split}
	D_1V_j(K_N)&\ll \biggl({\frac{\log N}{N}}\biggr)^{\frac{j+1}{n+1}}\, \biggl({\frac{\log N}{N}}\biggr)^{\frac{n-j}{n+1}}\,\mathbf{1}\{X_1\in K\setminus K_{(c\log N/N)}\}\\
	&= {\frac{\log N}{N}}\,\mathbf{1}\{X_1\in K\setminus K_{(c\log N/N)}\}\,.
	\end{split}
	\end{equation}
	
	To extend the argument for the general case, we argue as in \cite[Section~6]{BaranyFodorVigh}. Namely, since $K$ is compact, we can choose $\gamma\in(0,\infty)$ and $\Gamma\in(0,\infty)$ to be, respectively, the global lower and the global upper bound on the principal curvatures of $\partial K$. Remark 5 on page 126 of \cite{Schneider} ensures that under our assumptions on the smoothness of the convex body $K$ all projected images of $K$ also have a boundary with the same features as $\partial K$, and we choose $\gamma$ and $\Gamma$ such that they also bound from below and above the principal curvatures of each $j$-dimensional projection of $K$. Since we can approximate \( \partial K\) locally with affine images of balls, the construction of the cap \(C\) above and the relations regarding its volume, its central angle and the subspaces \(L\) which ensure \(C|L\subseteq K_{N-1}|L\) are not affected. Due to this, the relations
	$$
	\mathrm{vol}_j(C|L)\ll (\log N/N)^{\frac{j+1}{n+1}}\,,\quad \alpha\ll V_n(C)^\frac{1}{n+1}\ll (\log N/N)^{\frac{1}{n+1}}
	$$
	and
	$$
	\sphericalangle(z,L)\ll\alpha
	$$
	from the above argument still hold, but this time the implicit constants depend on $\g$ and $\G$. From here, the bound \eqref{eq:D1bound} can be obtained in the same way as for the ball.

	Moreover, on the complement $B_1^c$ of $B_1$, we use the trivial estimate $D_1V_j(K_N)$ $\leq V_j(K)$ and thus conclude that
	\begin{equation*}
	\begin{split}
	\EE[(D_1V_j(K_N))^p] &= \EE[(D_1V_j(K_N))^p\,\mathbf{1}_{B_1}]+\EE[(D_1V_j(K_N))^p\,\mathbf{1}_{B_1^c}]\\
	&\ll \biggl({\frac{\log N}{N}}\biggr)^{p}\,V_n(K\setminus K_{(c\log N/N)})\ll \biggl({\frac{\log N}{N}}\biggr)^{p+\frac{2}{n+1}}
	\end{split}
	\end{equation*}
	for all $p\in\{1,2,3,4\}$, where we applied the probability estimate \eqref{eq:ProbabilityAc} in the second step, which ensures that the second term can be made very small for large $N$ (the choice for $p$ is motived by our applications below). As a consequence, we can already bound the terms appearing in the normal approximation bound in Lemma \ref{lem:SecondOrderPoincare} that involve $\g_3$ and $\g_4$. Namely, using the lower variance bounds \eqref{eq:LowerVarianceBound} we see that
	\begin{align*}
	\frac{\sqrt{N}}{\Var V_j(K_N)}\,\sqrt{\g_3} &\ll\frac{\sqrt{N}}{ N^{-\frac{n+3}{n+1}}}\,\biggl({\frac{\log N}{N}}\biggr)^{2+\frac{1}{n+1}} = N^{-\frac{1}{2}+\frac{1}{n+1}}\,(\log N)^{2+\frac{1}{n+1}}\,,\\
	\frac{N}{(\Var V_j(K_N))^\frac{3}{2}}\,\g_4 & \ll\frac{N}{N^{-\frac{3}{ 2}\frac{n+3}{n+1}}}\,\biggl({\frac{\log N}{N}}\biggr)^{3+\frac{2}{n+1}} = N^{-\frac{1}{2}+\frac{1}{n+1}}\,(\log N)^{3+\frac{2}{n+1}}\,.
	\end{align*}
	Next, we consider the second-order difference operator. 
	For \(z\in K\setminus K_{(c\log N/N)}\), recall that
	\begin{equation*}
	\Vis_z(N) := \{x\in K\setminus K_{(c\log N/N)}:[x,z]\cap K_{(c\log N/N)}=\emptyset\}\,.
	\end{equation*}
	On the event $B_2$ it may be concluded from \eqref{eq:D1bound} that $D_if(V)^2\ll (\log N/N)^{2}$ for all $i\in\{1,2,3\}$ and $V\in\{Z,Z'\}$. 
	We note that on $B_2$ the following inclusion holds:
	\begin{align*}
	\{D_{1,2}f(Y)\neq 0\}&\subseteq\{Y_1\in K\setminus K_{(c\log N/N)}\}\cap \{Y_2\in K\setminus K_{(c\log N/N)}\}\\
	&\quad\quad\cap \{\mathrm{Vis}_{Y_1}(N)\cap\mathrm{Vis}_{Y_2}(N)\neq\emptyset\}\\
	&\subseteq \{Y_1\in K\setminus K_{(c\log N/N)}\}\cap\biggl\{Y_2\in\bigcup_{x\in\mathrm{Vis}_{Y_1}(N)} \mathrm{Vis}_x(N)\biggr\}\,.
	\end{align*}
	The same applies to $D_{1,3}f(Y')$ as well. We thus infer that
	\begin{align*}
	&\EE\bigl[\mathbf{1}\{D_{1,2}f(Y)\neq 0\}\,\mathbf{1}_{B_2}\bigr]\\
	&\le\PP\bigl( Y_1\in K\setminus K_{(c\log N/N)}\bigr)\,\PP \biggl( Y_2\in\bigcup_{x\in\mathrm{Vis}_{Y_1}(N)} \mathrm{Vis}_x(N) \biggm| Y_1\in K\setminus K_{(c\log N/N)} \biggr)\\
	&\le \PP\bigl(Y_1\in K\setminus K_{(c\log N/N)}\bigr)\,\sup_{z\in K\setminus K_{(c\log N/N)}}\PP \biggl( Y_2\in\bigcup_{x\in\mathrm{Vis}_{z}(N)} \mathrm{Vis}_x(N)\biggr)\\
	&=V_n \Bigl(K\setminus K_{(c\log N/N)}\bigr)\,\sup_{z\in K\setminus K_{(c\log N/N)}} V_n\Bigl(\bigcup_{x\in\mathrm{Vis}_{z}(N)} \mathrm{Vis}_x(N) \Bigr)\,.
	\end{align*}
	Since the diameter of the previous union is of order $(\log N/N)^{\frac{1}{n+1}}$, it follows from \cite[Lemma 6.2]{VuConcentration} that
	\begin{equation*}
	\Delta(N):=\sup_{z\in K\setminus K_{(c\log N/N)}}V_n\Bigl(\bigcup_{x\in\mathrm{Vis}_z(N)} \mathrm{Vis}_x(N)\Bigr)\ll {\frac{\log N}{N}}\,.
	\end{equation*}		
	Moreover, on the complement $B_2^c$ of $B_2$ we estimate all the indicator functions by one and the value of all difference operators by the constant $V_j(K)$. Since $\PP(B_2^c)$ is small in $N$ (recall \eqref{eq:ProbabilityBc}), this readily implies
	\begin{equation*}
	\g_2 \ll \biggl({\frac{\log N}{N}}\biggr)^{4}\,V_n(K\setminus K_{(c\log N/N)})\,\Delta(N)\ll \biggl({\frac{\log N}{N}}\biggr)^{5+\frac{2}{n+1}}\,.
	\end{equation*}
	Analogously, we can bound \(\gamma_1\). First, suppose that \(Y_1=Y_1'\). Then, conditioned on $B_2$,
	\begin{align*}
	&\{D_{1,2}f(Y)\neq 0\}\cap\{D_{1,3}f(Y')\neq 0\}\\
	&\subseteq\{\{Y_1,Y_2,Y_3'\}\subseteq K\setminus K_{(c\log N/N)}\}\cap \{\mathrm{Vis}_{Y_2}(N)\cap\mathrm{Vis}_{Y_1}(N)\neq\emptyset\}\\
	&\quad\quad\cap \{\mathrm{Vis}_{Y_3'}(N)\cap\mathrm{Vis}_{Y_1}(N)\neq\emptyset\}\\
	&\subseteq \{Y_1\in K\setminus K_{(c\log N/N)}\}\cap\biggl\{\{Y_2,Y_3'\}\subseteq\bigcup_{x\in\mathrm{Vis}_{Y_1}(N)} \mathrm{Vis}_x(N)\biggr\}\,,
	\end{align*}
	and arguing as before leads to
	\begin{align*}
	&\EE\bigl[\mathbf{1}\{D_{1,2}f(Y)\neq 0\} \,\mathbf{1}\{D_{1,3}f(Y')\neq 0\}\,\mathbf{1}_{B_2}\bigr]\\
	&\leq 
	\PP\bigl( Y_1\in K\setminus K_{(c\log N/N)}\bigr)\sup_{z\in K\setminus K_{(c\log N/N)}}\PP \biggl( \{Y_2,Y_3'\}\subseteq\!\!\!\bigcup_{x\in\mathrm{Vis}_{z}(N)} \!\!\!\mathrm{Vis}_x(N)\biggr)\\
	&\le V_n(K\setminus K_{(c\log N/N)})\,\Delta(N)^2\,.
	\end{align*}
	Note that the case \(Y_1\neq Y_1'\) gives a smaller order since, by independence, it leads to an extra factor \(V_n(K\setminus K_{(c\log N/N)})\). Thus, by conditioning on $B_2$ and its complement, we obtain
	\begin{equation*}
	\g_1 \ll \biggl({\frac{\log N}{N}}\biggr)^{4} \,V_n(K\setminus K_{(c\log N/N)})\,\Delta(N)^2 \ll \biggl({\frac{\log N}{N}}\biggr)^{6+\frac{2}{n+1}}\,.
	\end{equation*}
	Now, the other terms appearing in the normal approximation bound in Lemma \ref{lem:SecondOrderPoincare} can be estimated using the lower variance bounds \eqref{eq:LowerVarianceBound} as follows,
	\begin{align*}
	\frac{\sqrt{N}}{\Var V_j(K_N)}\,\sqrt{N^2\g_1} &\ll\frac{\sqrt{N}}{ N^{-\frac{n+3}{n+1}}}\,\sqrt{N^2\cdot \biggl({\frac{\log N}{N}}\biggr)^{6+\frac{2}{n+1}}} = N^{-\frac{1}{2}+\frac{1}{n+1}}\,(\log N)^{3+\frac{1}{n+1}}\,,\\
	\frac{\sqrt{N}}{\Var V_j(K_N)}\,\sqrt{N\g_2} &\ll \frac{\sqrt{N}}{ N^{-\frac{n+3}{n+1}}}\,\sqrt{N\cdot \biggl({\frac{\log N}{N}}\biggr)^{5+\frac{2}{n+1}}} = N^{-\frac{1}{2}+\frac{1}{n+1}}\,(\log N)^{\frac{5}{2}+\frac{1}{n+1}}\,.
	\end{align*}
	Putting together all estimates, we arrive at
	\begin{equation}\label{eq:RateOfConvergence}
	\begin{split}
	\mathrm{d}\biggl(\frac{V_j(K_N)-\EE V_j(K_N)}{\sqrt{\Var V_j(K_N)}},G\biggr) &\ll N^{-\frac{1}{2}+\frac{1}{n+1}}\Bigl((\log N)^{3+\frac{1}{n+1}}+(\log N)^{\frac{5}{ 2}+\frac{1}{n+1}}\\
	&\qquad\qquad\qquad+(\log N)^{2+\frac{1}{n+1}}+(\log N)^{3+\frac{2}{n+1}}\Bigr)\\
	&\ll N^{-\frac{1}{2}+\frac{1}{n+1}}\,(\log N)^{3+\frac{2}{n+1}}
	\end{split}
	\end{equation}
	in view of the normal approximation bound in Lemma \ref{lem:SecondOrderPoincare}. In particular, as $N\to\infty$, this converges to zero and so the random variables
	\begin{equation*}
	W_j(K_N)=\frac{V_j(K_N)-\EE V_j(K_N)}{\sqrt{\Var V_j(K_N)}}
	\end{equation*}
	converge in distribution to the standard Gaussian random variable $G$. The proof of Theorem \ref{thm:CLT} is thus complete.
\end{proof}	
\subsection*{Acknowledgement}
We would like to thank Matthias Schulte for pointing us to a mistake in the first version of this manuscript, as well as the two referees for their suggestions.

\bibliographystyle{amsplain}

\end{document}